\RequirePackage{fix-cm}
\documentclass[oneside,british,english]{amsart}
\usepackage{graphicx}
\usepackage[T1]{fontenc}
\usepackage{float}
\usepackage{pdfpages}
\usepackage{slashed}
\usepackage{algorithm2e}
\usepackage{amstext}
\usepackage{amsthm}
\usepackage{amssymb}
\usepackage{stmaryrd}
\usepackage{wasysym}
\usepackage{algorithmic}
\usepackage{matlab-prettifier}
\usepackage{enumitem}
\newtheorem*{theorem-non}{Theorem}
\newtheorem{thm}{Theorem}
\newtheorem{lem}[thm]{Lemma}
\newtheorem{cor}[thm]{Corollary}

\begin{document}

\title{A convex cover for closed unit curves has area at least 0.0975}


\author{Bogdan Grechuk         \and
        Sittichoke Som-am 
}




\maketitle

\begin{abstract}
We combine geometric methods with numerical box search algorithm to show that the minimal area of a convex set on the plane which can cover every closed plane curve of unit length is at least $0.0975$. This improves the best previous lower bound of $0.096694$. In fact, we show that the minimal area of convex hull of circle, equilateral triangle, and rectangle of perimeter $1$ is between $0.0975$ and $0.09763$.
\end{abstract}

\section{Introduction}
\label{intro}
In 1966, Leo Moser \cite{moser1966poorly} posed a question to determine the smallest area of plane region $S$
which can cover every curve of length $1$, assuming that the curve may be translated and rotated to fit inside the region. This problem, known as ``Moser's worm problem'', remains open, and it is not even known that the smallest-area region with this property exists - an alternative possibility is that the infimum of possible areas of $S$ can not be actually attained. 
The current record belonging to Norwood and Poole \cite{norwood2003improved}, who constructed a (nonconvex) set $S$ with this property and area at most $0.260437$. From the lower bounds perspective, it is known only that area of any such set $S$ must be strictly positive \cite{marstrand1979packing}. 

Several variants of Moser's worm problem has been studied in the literature. In particular, one can modify:
\begin{enumerate}[label=(\arabic*)]
\item the allowed covering regions (e.g. triangle, rectangle, convex, non-convex)
\item the sets of curves to be covered (e.g. closed curves, convex curves)
\end{enumerate}

If one insists for the covering set $S$ to be convex, Laidecker and Poole \cite{laidacker1986existence} used Blachke Selection Theorem to show that the solution to Moser's worm problem exists. 
If $\alpha$ is the area of this smallest convex covering set, it is known that $0.232239\leq\alpha\leq 0.27091$, where the upper bound  was found by Wang \cite{wang2006improved} in 2006, while the lower bound was proved by Khandhawit, Sriswasdi and Pagonakis \cite{khandhawit2013lower} in 2013. 

This work studies the version of Moser's worm problem, asking for {\bf \emph{convex}} region $S$ of minimal area $\beta$, which can cover every {\bf \emph{closed}} curve of length $1$. H.G. Eggleton \cite{eggleston2007problems} proved in 1957 that a triangle covers all closed unit curves if and only if it can cover the circle of perimeter $1$. Therefore, the smallest triangle which covers every closed unit curve is the equilateral triangle of side ${\displaystyle \frac{\sqrt{3}}{\pi}}$ and area $\frac{3\sqrt{3}}{4\pi^2}\approx 0.13162$. In 1972, Shaer and Wetzel \cite{schaer1972boxes} proved that the smallest rectangle with this property has side lengths $\frac{1}{\pi}$ and $\sqrt{\frac{1}{4}-\frac{1}{\pi^2}}$ and area about $0.12274$.
In 2006, Furedi and Wetzel \cite{furedi2011covers} found a cover in the form of pentagon with area less than $0.11222$. Recently, Wichiramala \cite{wichiramala2018smaller} discovered a hexagonal cover with area less than $0.11023$.

Minimum-area convex cover for a set of curves is, equivalently, the minimum-area convex hull of these curves. In 1973, Chakerian and Klamkin \cite{chakerian1973minimal} proved that the convex hull of the circle with perimeter $1$ and the line segment of length $\frac{1}{2}$ is at least $0.0963275$, thus providing the first lower bound for $\beta$.
In 2006, Furedi and Wetzel \cite{furedi2011covers} proved that the convex hull of the circle with perimeter 1 and the $0.0261682\times0.4738318$ rectangle is at least $0.0966675$. In 2011, Furedi and Wetzel \cite{furedi2011covers} improved the lower bound by replacing the $0.0261682\times0.4738318$ rectangle by curvilinear rectangle. This gives the area of convex hull about $0.096694$. Hence, the best published bounds for $\beta$ were, before this work,
$$
0.096694 \leq \beta \leq 0.11023.
$$
In an unpublished work, Som-am \cite{som2010improved} used the Brass grid search method \cite{brass2005lower} to show that the minimal-area convex hull of the line segment of length $\frac{1}{2}$, circle with perimeter $1$, and the equilateral triangle with perimeter $1$, is about $0.096905$.

The main result of this paper is the following one.

\begin{thm}\label{th:main}
Any convex set $S$ on the plane which can cover circle of perimeter $1$, equilateral triangle of perimeter $1$, and rectangle of size $0.0375\times0.4625$ (and perimeter $1$) has area at least $0.0975$.
\end{thm}

Theorem \ref{th:main} immediately implies that  
$$
0.0975 \leq \beta,
$$
which is an improvement comparing the best published lower bound $0.096694$, as well as comparing an unpublished lower bound $0.096905$.

If $\beta'$ is the minimal area of a set which can cover circle, equilateral triangle, and \emph{any} rectangle of perimeter $1$, then Theorem \ref{th:main} states that 
$$
0.0975 \leq \beta' \leq \beta.
$$
Our computation shows that the \emph{actual} value of $\beta'$ is about
$$
\beta' \approx 0.09762742.
$$
The bound in Theorem \ref{th:main} is slightly weaker, because we need some margin to allow rigorous analysis of our numerical algorithms. We also show rigorously that $\beta' \leq 0.09763$. This implies that, to improve lower bound for $\beta$ to $0.09763$ and beyond, different approach is required.

This paper is organized as follows. In Section 2, we give some basic definitions, and prove preliminary lemmas by geometric methods. These lemmas, together with the box-search numerical algorithm, are applied for proving Theorem \ref{th:main} in Section 3 and 4. Some concluding discussion is presented in Section 5.

\section{Geometric analysis}
\label{sec:1}
Assume that $C$ is a circle with radius $r={\displaystyle \frac{1}{2\pi}}$, $R$ is a rectangle with side $u\times v$ such that $u+v={\displaystyle \frac{1}{2}}$ and $u=0.0375$, and $T$ is an equilateral triangle of side ${\displaystyle \frac{1}{3}}$. Remark that $C$, $R$, and $T$ are convex polygons in $\mathbb{R}^{2}$. Our aim is to prove that, no matter how $C$, $R$, and $T$ are placed in $\mathbb{R}^{2}$, the area of their convex hull is at least $0.0975$.

Let $F$ be a regular $500$-gon inscribed in the circle, such that the sides of $R$ are parallel to some longest diagonals of $F$. We will call the union $X=F \cup R \cup T$ a \emph{configuration}.  
For any configuration $X$, let $\mathcal{H}(X)$ denote the convex hull of $X$, and $\mathcal{A}(X)$ the area of $\mathcal{H}(X)$.

Let us put a coordinate center $(0,0)$ at the center of $F$, and let X axis and Y axis be parallel to the longer and shorter sides of the rectangle, respectively. 
Let $R_0 (x_{1},y_{1})$ be the center of $R$. We can orient the axes in such a way that $x_1 \geq 0$ and $y_1 \geq 0$.
The vertices of $R$ are defined by 
$R_{1}\left(x_{1}-{\left(\frac{\frac{1}{2}-u}{2}\right)},
y_{1}+{\frac{u}{2}}\right)$, 
$R_{2}\left(x_{1}+{\left(\frac{\frac{1}{2}-u}{2}\right)},y_{1}+{\frac{u}{2}}\right)$, 
$R_{3}\left(x_{1}+{\left(\frac{\frac{1}{2}-u}{2}\right)},y_{1}-{\frac{u}{2}}\right)$, 
and $R_{4}\left(x_{1}-{\left(\frac{\frac{1}{2}-u}{2}\right)},y_{1}-{\frac{u}{2}}\right)$.
Let $T_0 (x_{2},y_{2})$ be the center of $T$ and $\theta$ be the angle between X axis and $T_0T_1$,
\begin{figure}
\begin{centering}
\includegraphics[scale=0.6]{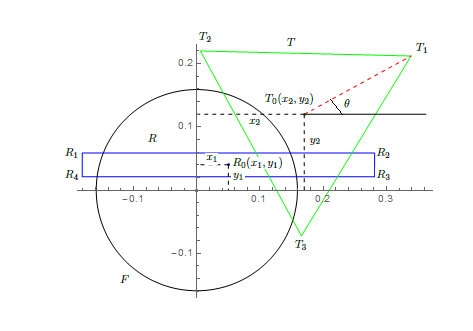}
\par\end{centering}

\caption{The configuration $X$ which depends on $x_1,y_1,x_2,y_2,\theta$}

\label{f1}
\end{figure}
Then $T_{1}(x_{2}+\frac{\sqrt{3}}{9}\cos\theta,y_{2}+\frac{\sqrt{3}}{9}\sin\theta),$
$T_{2}(x_{2}+\frac{\sqrt{3}}{9}\cos(\theta+\frac{2\pi}{3}),y_{2}+\frac{\sqrt{3}}{9}\sin(\theta+\frac{2\pi}{3}))$
and $T_{3}(x_{2}+\frac{\sqrt{3}}{9}\cos(\theta+\frac{4\pi}{3}),y_{2}+\frac{\sqrt{3}}{9}\sin(\theta+\frac{4\pi}{3}))$
are the vertices of triangle $T$.

In summary, the location of $F$, $R$, and $T$ is fully described by $5$ parameters: $x_{1},y_{1},x_{2},y_{2}$, and $\theta$.

Let $f:\mathbb{R}^{5}\rightarrow\mathbb{R}$ be a function which maps vector $(x_{1},y_{1},x_{2},y_{2},\theta)$ to the area $\mathcal{A}(X)$ of the convex hull of the corresponding configuration $X$. 
Clearly, $f$ is a continuous function. 
Because $F$ is a subset of $C$, Theorem \ref{th:main} would follow from the inequality 
$$
f(x_{1},y_{1},x_{2},y_{2},\theta) > 0.0975, \quad \forall \, x_{1}, y_{1}, x_{2}, y_{2}, \theta.
$$

The following result of Fary and Redei \cite{fary1950zentralsymmetrische} plays an important role in our analysis

\begin{lem}\label{lem:fary}\cite{fary1950zentralsymmetrische}
Let $S_1$ and $S_2$ be two bounded convex sets in ${\mathbb R}^2$. If $S_1$ is translated along a line with constant velocity, then the volume of the convex hull of $S_1$ and $S_2$ is a convex function of time.
\end{lem}

\begin{cor}\label{lem:coorconv}
Function $f$ is a convex function in each of the coordinates $x_{1},y_{1},x_{2},y_{2}$.
\end{cor}
\begin{proof}
Convexity of $f$ with respect to $x_1$ follows from Lemma \ref{lem:fary} with $S_2$ being the convex hull of $F$ and $T$, while $S_1=R$ moving along the X axis. Convexity of $f$ with respect to $y_1$, $x_2$, and $y_2$ follows from Lemma \ref{lem:fary} in a similar way.  
\end{proof}
\begin{lem}\label{lem:conditions}
Let $Z$ be a region of points $z=(x_{1},y_{1},x_{2},y_{2},\theta)$ in ${\mathbb R}^5$ satisfying the inequalities
$$
0\leq x_{1}\leq0.05,\,\, 0\leq y_{1}\leq0.04,\,\, -0.17\leq x_{2}\leq0.17,\,\, -0.13\leq y_{2}\leq 0.13,\,\, 0\leq\theta\leq{\displaystyle \frac{2\pi}{3}}.
$$
If $f(z) > 0.0975$ for all $z \in Z$, then in fact $f(z)> 0.0975$ for all $z \in {\mathbb R}^5$.
\end{lem}
\begin{proof}
Let $\psi(x_{1}, y_{1})$ be the area of convex hull of $F$ and $R$ only. Lemma \ref{lem:fary} implies that $\psi(x_{1}, y_{1})$ is a convex function in both coordinates. 
Assume that $x_1 \geq 0.05$.
By symmetry, $\psi(x_{1}, y_{1})=\psi(x_{1}, -y_{1})$, hence 
$$
\psi(x_{1}, y_{1}) \geq \psi(x_{1}, 0), \quad \forall x_{1}, y_{1}.
$$
Also, by symmetry, $\psi(x_{1}, 0)=\psi(-x_{1}, 0)$, hence convexity of $\psi(x_{1}, 0)$ implies that $\psi(x_{1}, 0) \geq \psi(0, 0)$, and that $\psi(x_{1}, 0)$ is non-decreasing in $x_1$ for $x_1 \geq 0$. Hence, $x_1 \geq 0.05$ implies that
$$
\psi(x_{1}, y_{1}) \geq \psi(x_{1}, 0) \geq \psi(0.05, 0) > 0.0975,
$$  
where the last equality is verified directly. For similar reasons,
$$
\psi(x_{1}, y_{1}) \geq \psi(0, y_{1}) \geq \psi(0, 0.04) > 0.0975,
$$ 
whenever $y_{1} \geq 0.04$.

From symmetry, we may assume that $x_1 \geq 0$ and $y_1 \geq 0$. Hence, either $f(z)> 0.0975$, or we may assume that $0\leq x_{1}\leq 0.05$, and $0\leq y_{1}\leq 0.04$.

Next, assume that $|x_{2}|\geq 0.17$. Then $\sqrt{x_{2}^{2}+y_{2}^{2}} \geq |x_{2}|\geq 0.17$. Let $C_{1}$ be the incircle of $T$ with radius $\frac{\sqrt{3}}{18}$ and center $(x_{2},y_{2})$, see Figure \ref{fig:f2}.
\begin{figure}[h]
\begin{centering}
\includegraphics[scale=0.4]{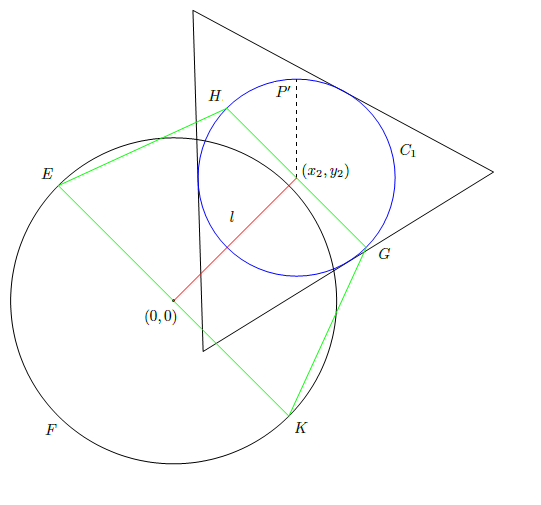}
\par\end{centering}

\caption{$\mathcal{A}(X)$ is bounded by the area of $EKGH, S(F)/2$ and the semicircle $C_{1}$}
\label{fig:f2}
\end{figure}

 Let $l$ be the line segment between $(0,0)$ and $(x_{2},y_{2})$. Next, let points $H,G \in C_1$ and $E,K \in F$ be such that line segments 
$HG$ and $EK$ are perpendicular to $l$, and pass through $(x_{2},y_{2})$ and $(0,0)$, respectively, see Figure \ref{fig:f2}. Then $EKGH$ is trapezoid with base lengths $|HG|=\frac{\sqrt{3}}{9}$ and $|EK|\geq 2r \cos\left(\frac{\pi}{500}\right)$, where $r=\frac{1}{2\pi}$. The area of $F$ is $S(F)=500\frac{r^2}{2}\sin\left(\frac{2\pi}{500}\right)$.
Thus, 
$$
\mathcal{A}(X)>{\displaystyle \frac{1}{2}\left(2r \cos\left(\frac{\pi}{500}\right)+\frac{\sqrt{3}}{9}\right)\left(\sqrt{x_{2}^{2}+y_{2}^{2}}\right)+\frac{S(F)}{2}+\frac{\pi\left(\frac{\sqrt{3}}{18}\right)^{2}}{2}} >0.0975
$$
To prove the bound for $y_2$, we need the following claim.

{\bf Claim 1.} If there is a point $P' \in T$ with $y$-coordinate $y^* \geq 0.13 +\frac{\sqrt{3}}{18}$, then $f(z) > 0.0975$.

Indeed, let $R'=(0,-r)$, and $R_{1},R_{2},R_{3},R_{4}$ be the vertices of the rectangle, see Figure \ref{fig:f3}. 

\begin{figure}[h]
\begin{centering}
\includegraphics[scale=0.5]{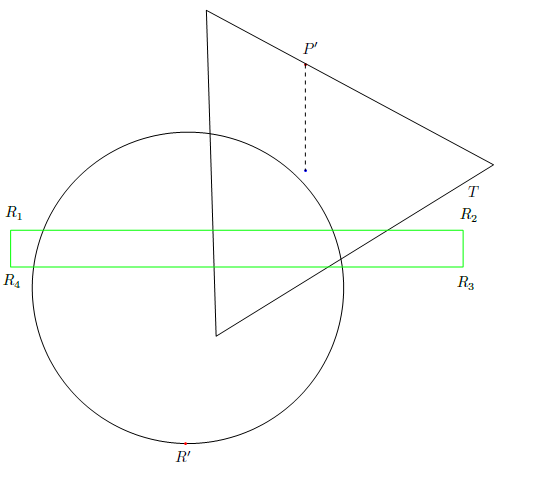}
\par\end{centering}

\caption{$\mathcal{A}(X)$ is bounded by the area of convex hull of $R',R_1,R_2,R_3,R_4$ and $P'$}
\label{fig:f3}
\end{figure}

Because $P',R_{1},R_{2},R_{3},R_{4},R'\in\mathcal{H}(X)$, we have  
$\mathcal{A}(X)>\mathcal{A}(\{P',R_{1},R_{2},R_{3},R_{4},R'\})$
$=u(\frac{1}{2}-u)+\frac{1}{2}(\frac{1}{2}-u)((y_2+\frac{\sqrt{3}}{18})+r-u)>0.0975$, and the claim follows.

Now, assume that $y_2 \geq 0.13$. Let $C_{1}$ be the same circle as above, see Figure \ref{fig:f2}, and let $P'$ be a point on $C_{1}$ with coordinates $(x_2, y_2+\frac{\sqrt{3}}{18})$. Then $P' \in T$, and $f(z) > 0.0975$ by the claim.
 
The cases $x_{2}\leq -0.17$ and $y_2 \leq -0.13$ are considered similarly.
\end{proof}
\begin{cor}\label{lem:bigrange}
Either $f(z)> 0.0975$, or $F \cup T \cup R$ is a subset of a rectangle with side lengths $0.386 \times 0.644$. 
\end{cor}
\begin{proof}
By Lemma \ref{lem:conditions}, we can assume that $z = (x_{1},y_{1},x_{2},y_{2},\theta) \in Z$. 
Let $Y_1$ and $Y_2$ be the points of configuration $X=F \cup T \cup R$ with the lowest and highest $y$-coordinates $y^*_1$ and $y^*_2$, respectively. Because $0\leq y_{1}\leq0.04$, $Y_1$ and $Y_2$ cannot belong to the rectangle $R$. If they both belong to $F$, then $y^*_2 - y^*_1 = \frac{1}{\pi} < 0.386$. If they both belong to the triangle, then $y^*_2 - y^*_1 \leq \frac{1}{3} < 0.386$. If $Y_1 \in F$ and $Y_2 \in T$, then $y^*_1 = - \frac{1}{2\pi}$, and the inequality $y^*_2 \leq 0.13+\frac{\sqrt{3}}{18}$ follows form the Claim 1 in the proof of Lemma \ref{lem:conditions}. Then $y^*_2-y^*_1 \leq  0.13+\frac{\sqrt{3}}{18} + \frac{1}{2\pi} < 0.386$.

Similarly, Let $X_1$ and $X_2$ be the points of configuration $X=F \cup T \cup R$ with the lowest and highest $x$-coordinates $x^*_1$ and $x^*_2$, respectively. $z \in Z$ implies that neither $X_1$ nor $X_2$ belongs to $F$. If $X_1 \in T$ and $X_2 \in R$, then, by Lemma \ref{lem:conditions} $x^*_1 \geq -0.17 - \frac{\sqrt{3}}{9}$, and $x^*_2 \leq 0.05 + \frac{0.4625}{2}$, hence $x^*_2 - x^*_1 \leq 0.05 + \frac{0.4625}{2} + 0.17 + \frac{\sqrt{3}}{9} < 0.644$. 
\end{proof}

The following lemma established Lipschitz continuity of $f$ in $Z$.

\begin{lem}\label{lem:cont}
For every $(x_{1},y_{1},x_{2},y_{2},\theta) \in Z$, and any $\epsilon_i \geq 0$, $i=1,\dots,5$,
$$
|f(x_{1}+\epsilon_{1},y_{1}+\epsilon_{2},x_{2}+\epsilon_{3},y_{2}+\epsilon_{4},\theta+\epsilon_{5})-f(x_{1},y_{1},x_{2},y_{2},\theta)| \leq \sum_{i=1}^5 \epsilon_i C_i,
$$

with constants 
$C_{1}=0.212$, $C_{2}=0.322$ , $C_{3}=0.326$ , $C_{4}=0.398$, and $C_{5}=0.134$.

\end{lem} 
\begin{proof}
If function $g:{\mathbb R}\to{\mathbb R}$ is convex on ${\mathbb R}$ and
$$
C = \max\left[\lim\limits_{t\to -\infty} \frac{g(t)}{t}, \lim\limits_{t\to +\infty} \frac{g(t)}{t}\right] < \infty,
$$
then 
$$
|g(t+\epsilon)-g(t)| \leq C \epsilon, \quad \forall t, \, \forall \epsilon > 0. 
$$
Indeed, inequality $C':=(g(t_0+\epsilon) - g(t_0))/\epsilon > C$ for some $t_0$ and $\epsilon$, would, by convexity of $g$, imply that $g(t_0+2\epsilon) > g(t_0) + 2C'\epsilon$, and, by induction, $g(t_0+2^n\epsilon) > g(t_0) + 2^nC'\epsilon$, a contradiction with $\lim\limits_{t\to +\infty} \frac{g(t)}{t} \leq C < C'$. Inequality $(g(t_0+\epsilon) - g(t_0))/\epsilon < -C$ leads to a contradiction for similar reasons.

Let us apply this result to convex function $g(x_1) = f(x_{1},y_{1},x_{2},y_{2},\theta)$, where $y_{1},x_{2},y_{2},\theta$ are fixed. 
In this case,
$$
\lim\limits_{x_1\to -\infty} \frac{g(x_1)}{x_1} = \lim\limits_{x_1\to +\infty} \frac{g(x_1)}{x_1} \leq \frac{0.386+0.0375}{2} < C_1,
$$
see Figure  \ref{fig:f4}.
\begin{figure}[h]
\begin{centering}
\includegraphics[scale=0.8]{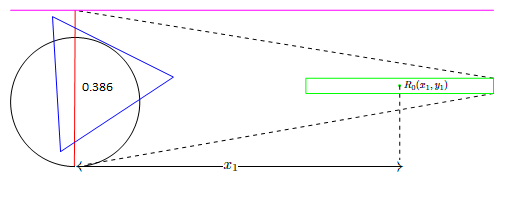}
\par\end{centering}

\caption{The ratio between $g(x_1)$ and $x_1$ when $x_1\to+\infty$}
\label{fig:f4}
\end{figure}

Hence,
$$
|f(x_{1}+\epsilon_{1},y_{1},x_{2},y_{2},\theta)-f(x_{1},y_{1},x_{2},y_{2},\theta)|\leq C_{1}\epsilon_{1}.
$$

Similarly, with $g(x_2) = f(x_{1},y_{1},x_{2},y_{2},\theta)$ for fixed $x_{1},y_{1},y_{2},\theta$, 
$$
\lim\limits_{x_2\to -\infty} \frac{g(x_2)}{x_2} = \lim\limits_{x_2\to +\infty} \frac{g(x_2)}{x_2} \leq \frac{2r+1/3}{2} < C_3,
$$
see Figure \ref{fig:f5},
while with $g(y_2) = f(x_{1},y_{1},x_{2},y_{2},\theta)$,
$$
\lim\limits_{y_2\to -\infty} \frac{g(y_2)}{y_2} = \lim\limits_{y_2\to +\infty} \frac{g(y_2)}{y_2} \leq \frac{1/3 + 0.4625}{2} < C_4.
$$
\begin{figure}[h]
\begin{centering}
\includegraphics[scale=0.7]{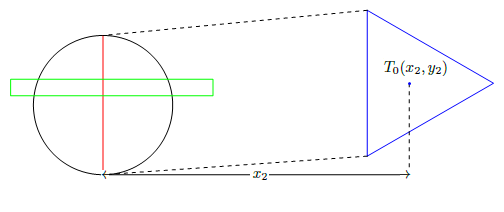}
\par\end{centering}

\caption{The ratio between $g(x_2)$ and $x_2$ when $x_2\to+\infty$}
\label{fig:f5}
\end{figure}

\begin{figure}[h]
\begin{centering}
\includegraphics[scale=0.5]{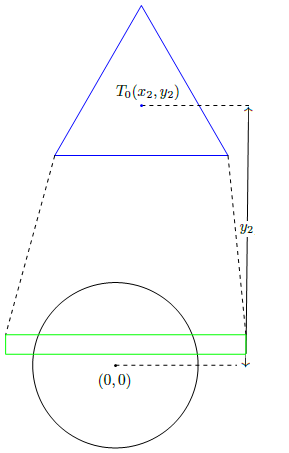}
\par\end{centering}

\caption{The ratio between $g(y_2)$ and $y_2$ when $y_2\to+\infty$}
\label{fig:f6}
\end{figure}

see Figure \ref{fig:f6}. This implies that
$$
|f(x_{1},y_{1},x_{2}+\epsilon_{3},y_{2},\theta)-f(x_{1},y_{1},x_{2},y_{2},\theta)|\leq C_{3}\epsilon_{3},
$$
and
$$
|f(x_{1},y_{1},x_{2},y_{2}+\epsilon_{4},\theta)-f(x_{1},y_{1},x_{2},y_{2},\theta)|\leq C_{4}\epsilon_{4}.
$$

The proof of similar bounds for the second and the fifth coordinates requires a different approach. For the second coordinate, we need to prove that 
\begin{equation}\label{eq:C2}
|f(x_{1},y_{1}+\epsilon_{2},x_{2},y_{2},\theta)-f(x_{1},y_{1},x_{2},y_{2},\theta)|\leq C_{2}\epsilon_{2}.
\end{equation} 
We claim that it is sufficient to prove (\ref{eq:C2}) for $\epsilon_{2}\in (0,\epsilon)$ for some $\epsilon>0$, which can depend on $x_{1},y_{1},x_{2},y_{2}$, and $\theta$. Indeed, assume, by contradiction, that, for some $y_1$, (\ref{eq:C2}) holds for $\epsilon_{2}\in (0,\epsilon)$ but not for all $\epsilon_2>0$. Let $\epsilon^*$ be the supremum of all $\epsilon$ such that (\ref{eq:C2}) holds for $\epsilon_{2}\in (0,\epsilon)$. Then, by continuity of $f$, (\ref{eq:C2}) also holds for $\epsilon_{2} = \epsilon^*$, that is,
$$
|f(x_{1},y_{1}+\epsilon^*,x_{2},y_{2},\theta)-f(x_{1},y_{1},x_{2},y_{2},\theta)|\leq C_{2}\epsilon^*.
$$
Applying (\ref{eq:C2}) to $y'_{1}=y_{1}+\epsilon^*$, we find that
$$
|f(x_{1},y_{1}+\epsilon^*+\delta_{2},x_{2},y_{2},\theta)-f(x_{1},y_{1}+\epsilon^*,x_{2},y_{2},\theta)|\leq C_{2}\delta_{2}
$$     
holds for all $\delta_2 \in (0, \delta)$ for some $\delta>0$. But the last two inequalities imply that (\ref{eq:C2}) holds for all $\epsilon_{2}\in (0,\epsilon^*+\delta)$, a contradiction with the definition of $\epsilon^*$. 

We next prove (\ref{eq:C2}) for $\epsilon_{2}\in (0,\epsilon)$. 
Let $R'$ with vertices $R'_1 R'_2 R'_3 R'_4$ be the rectangle $R$ which moved up by $\epsilon_2$ in Y-axis's direction. Convex hulls $\mathcal{H}(R,F,T)$ and $\mathcal{H}(R',F,T)$ are polygons, and, by selecting $\epsilon$ sufficiently small, we can assume that all vertices of these polygons, which are not vertices of $R$ and $R'$, coincides. Then $\mathcal{A}(R',F,T)-\mathcal{A}(R,F,T)$ is bounded by the total area of three triangles, say $R'_1 R_1 D_1, R'_2 R_2 D_2$, and $R'_2 R_2 T^{*}$ which $D_1,D_2\in F$ and $T^{*}\in T$ see Figure \ref{fig:f7}. Let $h_1,h_2,h_3$ be the height of $R'_1 R_1 D_1, R'_2 R_2 D_2$, $R'_2 R_2 T^{*}$, respectively. By Corollary \ref{lem:bigrange}, $|f(x_1,y_1+\epsilon_2,x_2,y_2,\theta)-f(x_1,y_1,x_2,y_2,\theta)|\leq \frac{1}{2}\epsilon_2(h_1+h_2+h_3)\leq\frac{1}{2}\epsilon_2(0.644)=0.322\epsilon_2 = C_2\epsilon_2$. 

\begin{figure}[h]
\begin{centering}
\includegraphics[scale=0.5]{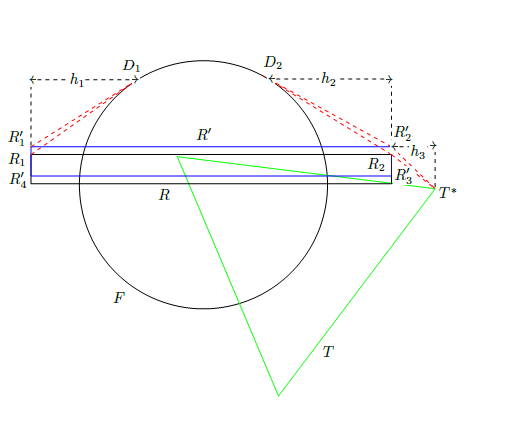}
\par\end{centering}

\caption{The three triangles which increase the area of convex hull with $\epsilon_2$}
\label{fig:f7}
\end{figure}
Finally, we need to prove that 
\begin{equation}\label{eq:C5}
|f(x_{1},y_{1},x_{2},y_{2},\theta+\epsilon_{5})-f(x_{1},y_{1},x_{2},y_{2},\theta)|\leq C_{5}\epsilon_{5}.
\end{equation} 
To prove the bound for $C_5$, we need the following claim.

{\bf Claim 2.} The diameter $d(\mathcal{F\cup R})$ of $\mathcal{F\cup R}$ is less than $0.46402$. 

Indeed, let $R_0=(0.05,0.04)$. We get $R_2=(0.28125,0.05875)$. Let $F_1\in F$ be a point which $d(x,R_2)$ is maximum for all $x\in F$ see Figure \ref{fig:f8}. By direct calculation, $R_2 F_1=0.4465 < |R_2 R_4|$. Hence, the diameter of $\mathcal{F\cup R}$ is $|R_2 R_4| < 0.46402$.
\begin{figure}[h]
\begin{centering}
\includegraphics[scale=0.5]{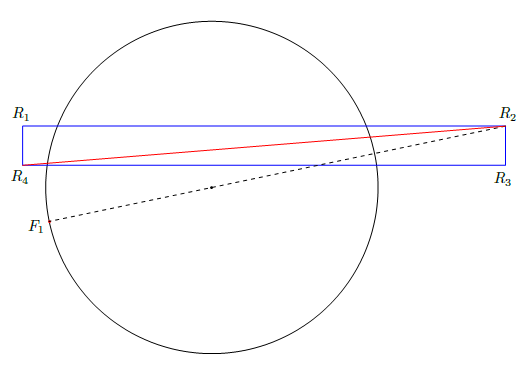}
\par\end{centering}

\caption{The longest between $R_2$ and $F$}
\label{fig:f8}
\end{figure}

Next, we will prove (\ref{eq:C5}). 

Let $T'$ with vertices $T'_1, T'_2, T'_3$ be the triangle $T$ rotated around $T_0$ by angle $\epsilon_{5}$. Then $|T_1T'_1| = 2|T_0T_1|\sin(\epsilon_5/2) < 2|T_0T_1|(\epsilon_5/2)=|T_0T_1|\epsilon_5 = \frac{\sqrt{3}}{9}\epsilon_5$. Similarly, $|T_2T'_2|=|T_3T'_3|< \frac{\sqrt{3}}{9}\epsilon_5$.

By selecting $\epsilon_5$ sufficiently small, we can ensure that all vertices of polygons $\mathcal{H}(R,F,T)$ and $\mathcal{H}(R,F,T')$ coincides, except possibly the vertices of $T$ and $T'$.

We assume that no vertices of the triangle $T$ are adjacent in the convex hull of $\mathcal{H}(R,F,T)$. 
Let $X_1$ and $X_2$ be the vertices of $\mathcal{H}(R,F,T)$ adjacent to $T_1$, $X_3$ and $X_4$ - vertices of $\mathcal{H}(R,F,T)$ adjacent to $T_2$, and $X_5$ and $X_6$ - vertices of $\mathcal{H}(R,F,T)$ adjacent to $T_3$, see Figure \ref{fig:f9}. Let us denote $S(ABC)$ the area of any triangle $ABC$.

Then area difference $|\mathcal{A}(R,F,T')-\mathcal{A}(R,F,T)|$ is equal to 
$$
|(S(T_1X_1X_2) + S(T_2X_3X_4) + S(T_3X_5X_6)) - (S(T'_1X_1X_2) + S(T'_2X_3X_4) + S(T'_3X_5X_6))|.
$$ 
But $|(S(T_1X_1X_2) - S(T'_1X_1X_2)| = |\frac{1}{2}h_1|X_1X_2| - \frac{1}{2}h_2|X_1X_2|| = \frac{1}{2}|X_1X_2|\cdot |h_1-h_2|$, where $h_1$ and $h_2$ are heights of triangles $T_1X_1X_2$ and $T'_1X_1X_2$, respectively, see Figure \ref{fig:f10}. But $|X_1X_2| < 0.46402$ by Claim 2, and
$|h_1-h_2| \leq |T_1T'_1| < \frac{\sqrt{3}}{9}\epsilon_5$, hence $|(S(T_1X_1X_2) - S(T'_1X_1X_2)| < 0.46402 \cdot \frac{\sqrt{3}}{18}\epsilon_5$. The same bound holds for $|(S(T_2X_3X_4) - S(T'_2X_3X_4)|$ and $|(S(T_3X_5X_6) - S(T'_3X_5X_6)|$. Hence,  
$$
|\mathcal{A}(R,F,T')-\mathcal{A}(R,F,T)| < 3 \cdot 0.46402 \cdot \frac{\sqrt{3}}{18}\epsilon_5 < 0.134\epsilon_5 = C_5 \epsilon_5.
$$

\begin{figure}
\begin{centering}
\includegraphics[scale=0.6]{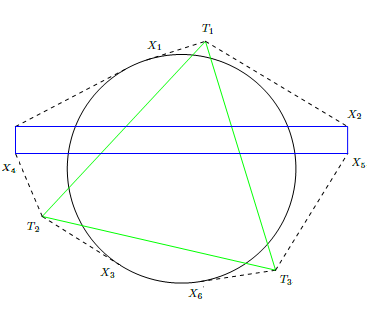}
\par\end{centering}

\caption{Polygon $\mathcal{H}(R, F, T)$ adjacent $T_1,T_2,T_3$}
\label{fig:f9}
\end{figure}

\begin{figure}
\begin{centering}
\includegraphics[scale=0.9]{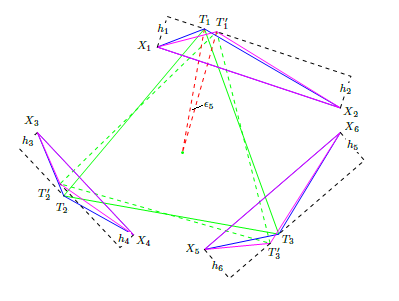}
\par\end{centering}

\caption{Six triangles which $T$ rotated by angle $\epsilon_5$}
\label{fig:f10}
\end{figure}

\end{proof}

\section{Computational results}
\label{sec:2}
In this section, we use box-search algorithm to prove that the minimal value of function $f(z)=f(x_{1},y_{1},x_{2},y_{2},\theta)$ in region $Z$ defined in Lemma \ref{lem:conditions} is grater than $0.0975$.

In general, let $B$ be a box $[a_{1},b_{1}]\times[a_{2},b_{2}]\times[a_{3},b_{3}]\times[a_{4},b_{4}]\times[a_{5},b_{5}]$ (that is, set of points $z=(x_{1},y_{1},x_{2},y_{2},\theta)$ such that $a_1 \leq x_1 \leq b_1$, $a_2 \leq y_1 \leq b_2$, $a_3 \leq x_2 \leq b_3$, $a_4 \leq y_2 \leq b_4$, $a_5 \leq \theta \leq b_5$. Let $z^*=(\frac{a_1+b_1}{2}, \frac{a_2+b_2}{2} , \frac{a_3+b_3}{2}, \frac{a_4+b_4}{2}, \frac{a_5+b_5}{2})$ be the center of the box). Then, if
\begin{equation}\label{eq:maincond}
f(z^*)-d_{1}C_{1}-d_{2}C_{2}-d_{3}C_{3}-d_{4}C_{4}-d_{5}C_{5}\geq 0.0975,
\end{equation}
where $d_i=\frac{b_{i}-a_{i}}{2}, \, i=1,2,3,4,5$, then, by Lemma \ref{lem:cont}, $f(z) \geq 0.0975$ for all $z\in B$.

If the condition (\ref{eq:maincond}) does not hold for $B$, we will divide $B$ into two sub-boxes $B_1$ and $B_2$, by dividing its maximal edge by half. For example, if $b_2-a_2 \geq b_i - a_i$, $i=1,3,4,5$, then edge with length $b_2-a_2$ is the maximal one, and we divide $B$ into 
$B_1=[a_{1},b_{1}]\times[a_{2},(a_2+b_2)/2]\times[a_{3},b_{3}]\times[a_{4},b_{4}]\times[a_{5},b_{5}]$
and
$B_2=[a_{1},b_{1}]\times[(a_2+b_2)/2, b_2]\times[a_{3},b_{3}]\times[a_{4},b_{4}]\times[a_{5},b_{5}]$.

We then check (\ref{eq:maincond}) for $B_1$ and $B_2$. If it holds in both cases, then $f(z) \geq 0.0975$ for all $z\in B_1$ and for all $z \in B_2$, hence $f(z) \geq 0.0975$ for all $z \in B$. If (\ref{eq:maincond}) does not hold for $B_1$ (or for $B_2$, of for both), we divide the corresponding box by two sub-boxes, and proceed iteratively.

$a_{1}=0,b_{1}=0.05,a_{2}=0,b_{2}=0.04,a_{3}=-0.17,b_{3}=0.17,a_{4}=-0.13,b_{4}=0.13,a_{5}=0,b_{5}=\frac{2\pi}{3}$. Then we evaluate $f$ in the box center $z^*=(0.025, 0.02, 0, 0, \pi/3)$ to check whether (\ref{eq:maincond}) holds. In this case, (\ref{eq:maincond}) reduces to
$$
f(z^*) - \frac{0.05}{2}C_1 - \frac{0.04}{2}C_2 - \frac{0.17-(-0.17)}{2}C_3 - \frac{0.13-(-0.13)}{2}C_4 - \frac{2\pi/3}{2}C_5  > 0.0975, 
$$
or equivalently, to $f(z^*) > 0.3567$.
However, the computation show that $f(z^*) \approx 0.10605 < 0.3567$, hence (\ref{eq:maincond}) does not hold. Hence, we need to subdivide $B$ into $B_1$ and $B_2$. In this case, $b_1-a_1=0.05$, $b_2-a_2=0.04$, $b_3-a_3=0.34$, $b_4-a_4=0.26$, and $b_5-a_5=2\pi/3 \approx 1.047$. Hence, $b_5-a_5 > b_i-a_i$, $i=1,2,3,4$, and we divide $B$ into
$B_1=[a_{1},b_{1}]\times[a_{1},b_{2}]\times[a_{3},b_{3}]\times[a_{4},b_{4}]\times[a_5,(a_5+b_5)/2]$
and
$B_2=[a_{1},b_{1}]\times[a_{1},b_{2}]\times[a_{3},b_{3}]\times[a_{4},b_{4}]\times[(a_5+b_5)/2, b_5]$.
Then we repeat the above procedure for $B_1$ and for $B_2$, and proceed iteratively.

We use Matlab$\circledR$ R2016a to implement this algorithm, see Algorithm 1. The actual Matlab code is presented in the Appendix.

\begin{algorithm}

\textbf{Function} checkmin 
$(a_{1},b_{1},a_{2},b_{2},a_{3},b_{3},a_{4},b_{4},a_{5},b_{5},a,r,n,B)$

\textbf{Variables:}
$a_{1},b_{1},a_{2},b_{2},a_{3},b_{3},a_{4},b_{4},a_{5},b_{5}$ - bounds for the box we consider at this iteration. $a$ - the minimal area of the convex hull we have found so far. $r$ - the total area of all boxes for which inequality \ref{eq:maincond} is verified so far (this variable is needed to control the progress). $n$ - the number of iterations so far. $B=0.0975$ is the lower bound we are proving.

\textbf{Procedure:}\\

\noindent 
1: set the coordinates of the box center: 

\textbf{ }$x_{1}:=\frac{a_{1}+b_{1}}{2},$ $y_{1}:=\frac{a_{2}+b_{2}}{2},$ $x_{2}:=\frac{a_{3}+b_{3}}{2},$ $y_{2}:=\frac{a_{4}+b_{4}}{2}$, $\theta:=\frac{a_{5}+b_{5}}{2}$\\
2: set $d_{1}:=\frac{b_{1}-a_{1}}{2},$ $d_{2}:=\frac{b_{2}-a_{2}}{2},$ $d_{3}:=\frac{b_{3}-a_{3}}{2},$ $d_{4}:=\frac{b_{4}-a_{4}}{2},d_{5}:=\frac{b_{5}-a_{5}}{2}$ and set $d:=max\{d_{1},d_{2},d_{3},d_{4},d_{5}\}$\\
3: calculate convex hull area $S$ of the configuration describes by parameters $x_{1},y_{1},x_{2},y_{2},\theta$ using function cvh: Set $S=cvh(x_{1},y_{1},x_{2},y_{2},\theta)$.\\

\noindent 4:\textbf{$\,$if} $S<a$ set $a=S$ (update of the minimal area we have found so far) 

\noindent 5: \textbf{end} \textbf{if};\textbf{}

6:\textbf{$\,$if} $S-C_{1}\cdot d_{1}-C_{2}\cdot d_{2}-C_{3}\cdot d_{3}-C_{4}\cdot d_{4}-C_{5}\cdot d_{5}>B$ (this means that this box is verified)

\noindent 7:$\:$set $r:=r+32\cdot d_{1}\cdot d_{2}\cdot d_{3}\cdot d_{4}\cdot d_{5}$; (add area of this box to $r$)

\noindent 8:$\:$set $\:n=n+1$; (update count for the number of iterations)\\ 

9:\textbf{$\:$return;}\\

10:\textbf{ else} (this means that \ref{eq:maincond} does not hold and we should sub-divide box by two boxes and proceed iteratively)

\textbf{$\quad\quad$switch} $d$ \textbf{do}

$\qquad$\textbf{case} 1

$\qquad\,$\textbf{for} i=0 \textbf{to} 1 \textbf{do}

\noindent 11:$\qquad\;$checkmin$(a_{1}+d_{1}\cdotp i,a_{1}+d_{1}\cdotp(i+1),a_{2},b_{2},a_{3},b_{3},a_{4},b_{4},a_{5},b_{5},a,r,n,B)$

\noindent 12:$\qquad\,$\textbf{end do}

\noindent 13:$\qquad$\textbf{case} 2

\noindent 14:$\qquad\,$\textbf{for} i=0 \textbf{to} 1 \textbf{do}

\noindent 15:$\qquad\;$checkmin$(a_{1},b_{1},a_{2}+d_{2}\cdotp i,a_{2}+d_{2}\cdotp(i+1),a_{3},b_{3},a_{4},b_{4},a_{5},b_{5},a,r,n,B)$

\noindent 16:$\qquad\,$\textbf{end do}

\noindent 17:$\qquad$\textbf{case} 3

\noindent 18:$\qquad\,$\textbf{for} i=0 \textbf{to} 1 \textbf{do}

\noindent 19:$\qquad\;$checkmin$(a_{1},b_{1},a_{2},b_{2},a_{3}+d_{3}\cdotp i,a_{3}+d_{3}\cdotp(i+1),a_{4},b_{4},a_{5},b_{5},a,r,n,B)$

\noindent 20:$\qquad\,$\textbf{end do}

\noindent 21:$\qquad$\textbf{case} 4

\noindent 22:$\qquad\,$\textbf{for} i=0 \textbf{to} 1 \textbf{do}

\noindent 23:$\qquad\;$checkmin$(a_{1},b_{1},a_{2},b_{2},a_{3},b_{3},a_{4}+d_{4}\cdotp i,a_{4}+d_{4}\cdotp(i+1),a_{5},b_{5},a,r,n,B)$

\noindent 24:$\qquad\,$\textbf{end do}

\noindent 25:$\qquad$\textbf{case} 5

\noindent 26:$\qquad\,$\textbf{for} i=0 \textbf{to} 1 \textbf{do}

\noindent 27:$\qquad\;$checkmin$(a_{1},b_{1},a_{2},b_{2},a_{3},b_{3},a_{4},b_{4},a_{5}+d_{5}\cdotp i,a_{5}+d_{5}\cdotp(i+1),a,r,n,B)$

\noindent 28:$\qquad\,$\textbf{end do}

\noindent 29:$\quad\quad$\textbf{end switch}

\noindent 30: \textbf{end} \textbf{if};\textbf{}

\noindent 31: \textbf{if} rem(n,1000000) == 0  print $r/100$ and $n$ (once in million iterations display the progress so far: for what percentage of the initial box the inequality \ref{eq:maincond}) is verified).

\noindent 32: \textbf{end} \textbf{if};\textbf{}

\textbf{end Procedure}

\noindent \textbf{\caption{\textbf{Box search algorithm}}
}

\end{algorithm}

The programme successfully verified the inequality $f(z) > 0.0975$ for all $z \in Z$ after $n = 7,180,439,126$ iterations. The program actually returned the minimal area $0.09762$ for the optimal configuration with $x_{1}=0.0251,\,y_{1}=0.00258,\,x_{2}=0.0653,\,y_{2}=0.00542,\,\theta=0.07989$ see Figure \ref{fig:f11} and Table \ref{fig:table}.

Repeating the calculation for this particular configuration in Mathematica with actual circle instead of $500$-gon shows that the optimal convex hull area is about 
$$
S_{min} \approx 0.09762742.
$$

\section{Main Theorems}

\begin{theorem-non} (Theorem \ref{th:main} in the Introduction)
Any convex set $S$ on the plane which can cover circle of perimeter $1$, equilateral triangle of perimeter $1$, and rectangle of size $0.0375\times0.4625$ (and perimeter $1$) has area at least $0.0975$.
\end{theorem-non}
\begin{proof}
By numerical results, $f(z) > 0.0975$ for all $z \in Z$, where $Z$ is defined in Lemma \ref{lem:conditions}. By Lemma \ref{lem:conditions}, this implies that $f(z) > 0.0975$ for all $z \in {\mathbb R}^5$, hence $\mathcal{A}(F,R,T)>0.0975$. Because $500$-gon $F$ is the subset of circle $C$, $\mathcal{A}(C,R,T) \geq \mathcal{A}(F,R,T)>0.0975$.
\end{proof}

\begin{cor}\label{cor:main}
Any convex cover for closed unit curves has area of at least $0.0975$.
\end{cor}
\begin{proof}
Because every convex cover for closed unit curves should cover the circle $C$, equilateral triangle $T$, and rectangle $R$ of size $0.0375\times0.4625$, the claim follows from Theorem \ref{th:main}. 
\end{proof}

\begin{figure}[H]
\begin{centering}
\includegraphics[scale=0.55]{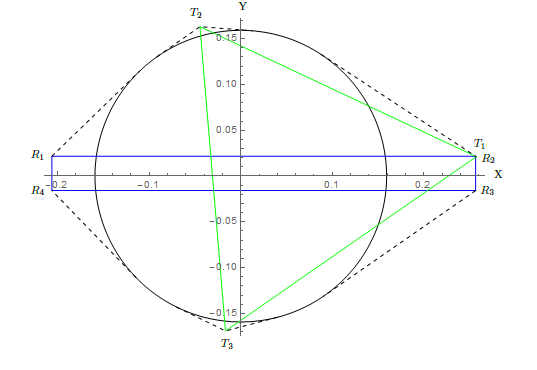}
\par\end{centering}

\caption{The convex hull of the configuration of the minimum
area with $0.097627$ acquired from the box-search algorithm}
\label{fig:f11}
\end{figure} 

The following Theorem implies that this method (with circle, equilateral triangle, and rectangle of perimeter $1$) cannot be used to improve the lower bound in Corollary \ref{cor:main} beyond $0.09763$. 

\begin{thm}\label{thm2}
For any rectangle $R'$ with perimeter $1$, there is a convex cover of $R'$, $C$, and $T$ with area at most $0.09763$.
\end{thm}
\begin{proof}
Let $l,w$ be the length and width of rectangle $R'$ such that $l+w=\frac{1}{2}$
and $w\in[0,0.25]$.\\
Let $F'$ be the regular 500-gon inscribed into the circle with $r'=\frac{sec(\frac{\pi}{500})}{2\pi}$. Then $C \subset F'$, and $\mathcal{H}(X)=\mathcal{H}(R,C,T)\subset\mathcal{H}(R,F',T)$. Thus,
$\mathcal{A}(X)\leq \mathcal{A}(R,F',T)$.

Let $f(w)$ denotes the minimal area of convex cover $R,F',T$.

\textbf{Claim} For any $\epsilon>0$, $|f(w+\epsilon)-f(w)|\leq 0.318\epsilon$.\\
It suffices to prove the claim only for ``small'' $\epsilon$. We will prove that $f(w)-f(w+\epsilon)\leq 0.318\epsilon$, the proof for inequality $f(w+\epsilon)-f(w)\leq 0.318\epsilon$ is similar. Let $R''$ be the rectangle with width $w+\epsilon$ and perimeter $1$.
Consider optimal configuration of $R'',F',T$, so that $f(\omega+\epsilon)=\mathcal{A}(R'',F',T)$.
Let us put $R'$ parallel to $R''$ as shown on Figure \ref{fig:f12}. This configuration is not necessary optimal, and, because $f$ denotes the area of the \emph{optimal} configuration, $f(w) \leq \mathcal{A}(R',F',T)$.  Hence, $f(w)-f(w+\epsilon)\leq \mathcal{A}(R',F',T)-\mathcal{A}(R'',F',T)$.

Convex hulls $\mathcal{H}(R',F',T)$ and $\mathcal{H}(R'',F',T)$ are polygons, and, by selecting $\epsilon$ sufficiently small, we can assume that all vertices of these polygons, which are not vertices of $R'$ and $R''$, coincides. Then $\mathcal{A}(R',F',T)-\mathcal{A}(R'',F',T)$ is bounded by the total area of triangles $XQ_1R_2$, $YQ_2R_3$, and rectangle $Q_1R_2R_3Q_2$, which is
$$
\frac{1}{2}h_1\epsilon + \frac{1}{2}h_2\epsilon + Q_1Q_2 \epsilon = \frac{\epsilon}{2}(h_1 + h_2 + 2Q_1Q_2) 
$$
We have $Q_1Q_2 = w \leq 0.25$, and, by Corollary \ref{lem:bigrange}, $h_1 + h_2 + Q_1Q_2 \leq 0.386$. Hence, 
$$
f(w)-f(w+\epsilon)\leq \mathcal{A}(R',F',T)-\mathcal{A}(R'',F',T) \leq \frac{\epsilon}{2}(0.386 + 0.25) = 0.318 \epsilon,
$$
which proves the claim.

To verify inequality $f(w)< 0.09763$ at some \emph{specific} point $w$, it is not necessary to find the \emph{optimal} configuration of $R'$, $F'$, and $T$. In suffices just to find \emph{some} configuration with $\mathcal{A}(R',F',T) < 0.09763$, and then conclude that $f(w) \leq \mathcal{A}(R',F',T) < 0.09763$. This makes the numerical verification simple.

We will verify inequality $f(w)< 0.09763$ for $w$ belonging to some finite set $W=\{w_1, w_2, \dots, w_N\}$, where $0 \leq w_1 \leq w_2 \leq \dots \leq w_N \leq 0.25$ are points to be specified below.
By the claim, inequality $f(w_i) < 0.09763$ implies that $f(w)\leq 0.09763$ in the whole interval $w \in [w_i - d_i, w_i + d_i]$, where $d_i = (0.09763 - f(w_i))/0.318$.

\begin{figure}[H]
\begin{centering}
\includegraphics[scale=0.35]{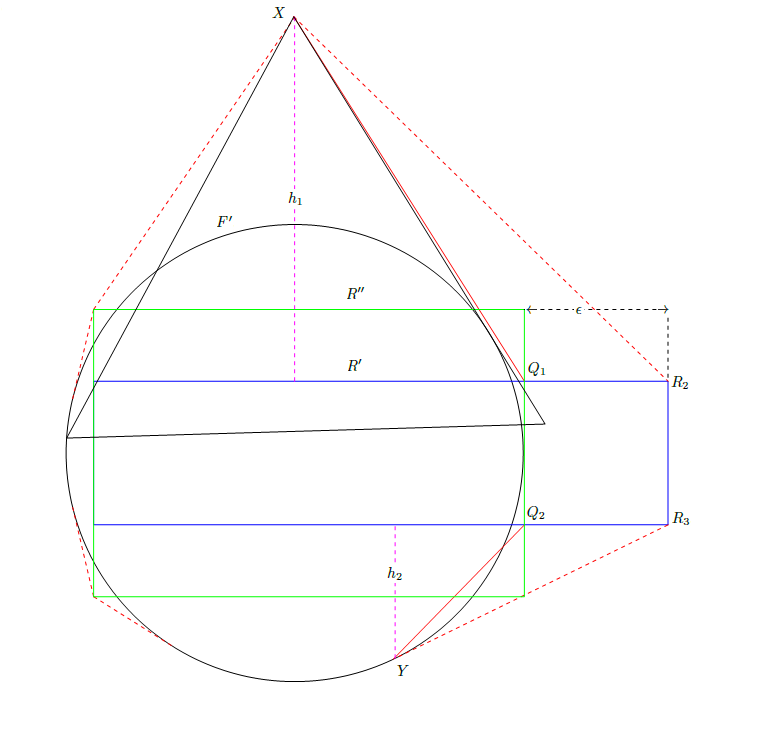}
\par\end{centering}

\caption{The configuration of $R$ and $R'$}
\label{fig:f12}
\end{figure}

We will select set $W$ in such a way that intervals $[w_i - d_i, w_i + d_i]$, $i=1,2,\dots,N$ cover the whole interval $[0, 0.25]$. In other words, $w_1 - d_1 < 0$, $w_N + d_N > 0.25$, and 
\begin{equation}\label{eq:cond2}
w_i + d_i < w_{i+1} - d_{i+1}, \quad i=1,2,\dots, N-1. 
\end{equation}

Set $W$ with $N=772$ points with this property is presented in the Appendix.
For example, $w_1=0.00020$, $w_2 = 0.0034$, $w_3=0.0086$, and so on, $w_{772}= 0.2415 $.

\end{proof}

\section{Conclusion and future work}

Using the method proposed in \cite{som2010improved}, we improved the lower bound of convex covers for closed unit curves from $0.096694$ to $0.0975$. The method is based on studying configurations of the regular 500-gon, rectangle, and equilateral triangle. The use geometric methods to prove Lipschitz bound for the corresponding function, and then use numerical box-search algorithm to finish the proof. 
Based on the numerical results, we conjecture that $T_{1}$ and $R_{2}$ coincides in the optimal configuration, see Figure \ref{fig:f13}. 
Our numerical results actually imply lower bound $0.09762742$, corresponding to the optimal configuration with parameters $x_{1}=0.0255904,y_{1}=0.0013503,x_{2}=0.0653055,y_{2}=0.0050124$
and $\theta=0.0766554$. However, we can formally prove only the weaker bound $0.0975$, using the algorithm which require $n = 7,180,439,126$ iterations. By using more powerful computers, the same methods can lead to, for example, bound $0.0976$. However, we have proved that one cannot achieve the bound $0.09763$ this way. An obvious way to improve the bound would be studying configurations with $4$ objects, but this would increase the dimension of the parameter space.

\begin{figure}[H]
\begin{centering}
\includegraphics[scale=0.7]{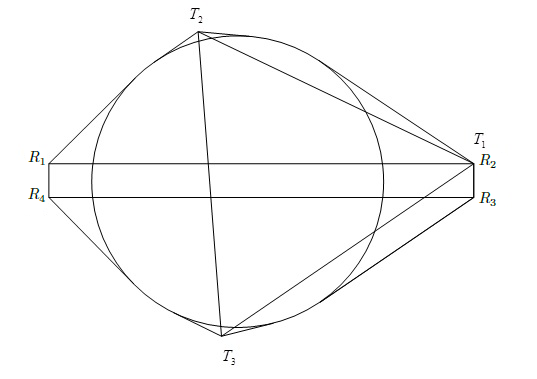}
\par\end{centering}

\caption{The minimum configuration which has an area of $0.09762742$ when
$T_{1}$ and $R_{2}$ coincide }
\label{fig:f13}
\end{figure}


\bibliographystyle{spbasic}      


\appendix

\section{Appendix}

\subsection{Box's search method}

Function $cvh (x_1, x_2, y_1, y_2, \alpha)$ below calculates the area of the convex hull for the configuration defined by parameters $x_1, x_2, y_1, y_2, \alpha$. Function $checkminNB4$ is used to check the main inequality \ref{eq:maincond} in a box domain defined by its parameters $a_1,b_1, \dots, a_5, b_5$.

\subsubsection{The area of convex hull fucntion}

\begin{lstlisting}
function g = cvh(x1,y1,x2,y2,alpha)
%cvh is an area of convex hull function
%x1, y1 are the center of a rectangle
%x2, y2 are the center of an equilateral triangle
%alpha is an angle of an equilateral triangle
v=0.0375; % the width of a rectangle
%create the 500-gon
t = linspace(0,2*pi,500); 
x=(cos(t)/(2*pi)); 
y=(sin(t)/(2*pi));
A=[x;y];
%create the vertex of a rectangle 
xR1=x1-(1/2-v)/2;
yR1=y1+v/2;
xR2=x1+(1/2-v)/2;
yR2=y1+v/2;
xR3=x1+(1/2-v)/2;
yR3=y1-v/2;
xR4=x1-(1/2-v)/2;
yR4=y1-v/2;
%create the vertex of an equilateral triangle
xQ1=x2+(sqrt(3)/9)*cos(alpha);
yQ1=y2+(sqrt(3)/9)*sin(alpha);
xQ2=x2+(sqrt(3)/9)*cos(alpha+(2*pi)/3);
yQ2=y2+(sqrt(3)/9)*sin(alpha+(2*pi)/3);
xQ3=x2+(sqrt(3)/9)*cos(alpha+(4*pi)/3);
yQ3=y2+(sqrt(3)/9)*sin(alpha+(4*pi)/3);
%B is the matrix of points of a rectangle
%C is the matrix of points of an equilateral triangle 
%D is the matrix of points of all objects
B=[xR1 xR2 xR3 xR4;yR1 yR2 yR3 yR4];
C=[xQ1 xQ2 xQ3;yQ1 yQ2 yQ3];
D= [A B C];
format long
%g is the area of convex hull of D 
[p,g]=convhull(D(1,:),D(2,:));
end
\end{lstlisting}

\subsubsection{The box's search method}

\begin{lstlisting}
(* ::Package:: *)

function [r,n,min,xx1,yy1,xx2,yy2,app] = checkminNB4(a1,b1,a2,b2,a3,
b3,a4,b4,a5,b5,s,ep1,area,count,st,ix1,iy1,ix2,iy2,iap)
% checkminNB4 is a function to check the main inequality
% Input
% a1 and b1 are a lower bound and an upper bound of x1
% a2 and b2 are a lower bound and an upper bound of y1
% a3 and b3 are a lower bound and an upper bound of x2
% a4 and b4 are a lower bound and an upper bound of y2
% a5 and b5 are a lower bound and an upper bound of alpha
% s is the bound which we want to prove 
% ep1 is tolerant
% area is a initial area of box (0)
% count is a initial times which satisfies the main inequality (0)
% st, ix1, iy1, ix2, iy2 and iap are the initial guess of smallest area,
% x1,y1,x2,y2 and alpha, respectively.
%Output
% r is a total area of box
% n is the number of times which satisfies the main inequality
% min, xx1, yy1, xx2, yy2 and aap are the smallest area,
% x1,y1,x2,y2 and alpha, respectively.

% define the box's point
x1=(a1+b1)/2;
y1=(a2+b2)/2;
x2=(a3+b3)/2;
y2=(a4+b4)/2;
alpha=(a5+b5)/2;

xx1=ix1;
yy1=iy1;
xx2=ix2;
yy2=iy2;
app=iap;
min=st;
% d_i is the length of the i box
d1=(b1-a1)/2;
d2=(b2-a2)/2;
d3=(b3-a3)/2;
d4=(b4-a4)/2;
d5=(b5-a5)/2;
n=count; 
r=area;
[d,di]=max([d1 d2 d3 d4 d5]); % find the max d_i and index
% the constants  
C1=0.212;
C2=0.322;
C3=0.326;
C4=0.398;
C5=0.134;

% check the tolerant
if d< ep1
    disp('error');
    r=-1;
end    


if cvh(x1,y1,x2,y2,alpha) - C1*d1 - C2*d2 - C3*d3 - C4*d4 - C5*d5 > s 
    r=r+32*d1*d2*d3*d4*d5; % find total area of box
    n=n+1;
    return;
   
else
   if cvh(x1,y1,x2,y2,alpha)< min % find the smallest area 
                    min=cvh(x1,y1,x2,y2,alpha);
                    xx1=x1;
                    yy1=y1;
                    xx2=x2;
                    yy2=y2;
                    app=alpha;
   end
% create a new box which has a maximum length
switch di
		case 1 % d1 is a max
            for i=0:1
             if r>=0 
[r,n,min,xx1,yy1,xx2,yy2,app] = checkminNB4(a1+d1 * i, a1+d1*(i+1), 
a2, b2,a3,b3, a4, b4,a5, b5,s,ep1,r,n,min,xx1,yy1,xx2,yy2,app);
             end
            end

        case 2 % d2 is a max
            for j=0:1
             if r>=0 
[r,n,min,xx1,yy1,xx2,yy2,app] = checkminNB4(a1, b1, a2+d2*j, 
a2+d2*(j+1),a3, b3, a4, b4,a5, b5,s,ep1,r,n,min,xx1,yy1,xx2,yy2,app); 
             end
            end
            
        case 3 % d3 is a max
            for k=0:1
             if r>=0 
[r,n,min,xx1,yy1,xx2,yy2,app] = checkminNB4(a1, b1, a2, b2, a3+d3*k,
 a3+d3*(k+1), a4, b4,a5, b5,s,ep1,r,n,min,xx1,yy1,xx2,yy2,app);              
             end
            end
            
        case 4 % d4 is a max
            for l=0:1
             if r>=0 
[r,n,min,xx1,yy1,xx2,yy2,app] = checkminNB4(a1, b1, a2, b2, a3, b3, 
a4+d4*l, a4 + d4*(l+1),a5, b5,s,ep1,r,n,min,xx1,yy1,xx2,yy2,app);
             end
            end
            
        case 5 % d5 is a max
            for m=0:1
             if r>=0 
[r,n,min,xx1,yy1,xx2,yy2,app] = checkminNB4(a1, b1, a2, b2, a3, b3,
 a4, b4,a5+d5*m, a5+d5*(m+1),s,ep1,r,n,min,xx1,yy1,xx2,yy2,app);
             end
            end
                
        otherwise
                error('Unknown method.');
    end
    
% print the results for every 1000000 steps   
if rem(n,1000000)==0
    disp([num2str ((100*r)/0.00036801) ' %' ' count =' num2str(n)]);
end

end
\end{lstlisting}

\subsubsection{The box's search results}

The Box search program displayed a message every $1,000,000$ steps. Figure \ref{fig:table} presents the output of these messages for (approximately) every $100,000,000$ steps. Here, the first column represents progress, in terms of the percentage of the area of the initial box for which the inequality \ref{eq:maincond} is verified. The second column is the iteration number. Figure \ref{fig:rngraph} presents the graphical illustration how progress depends on the number of iterations.

\begin{figure}
\caption{The table of the percentage of $r$ and $n$}\label{fig:table}
\begin{centering}
\begin{tabular}{|c|c|}
\hline 
Percentage of $r$ & $n$\tabularnewline
\hline 
\hline 
6.5703 \% & 100000000\tabularnewline
\hline 
6.5747 \%  & 200000000\tabularnewline
\hline 
6.5762 \%  & 300000000\tabularnewline
\hline 
6.5806 \%  & 400000000\tabularnewline
\hline 
6.5867 \%  & 500000000\tabularnewline
\hline 
6.5882 \%  & 600000000\tabularnewline
\hline 
6.6232 \%  & 702000000\tabularnewline
\hline 
6.6252 \%  & 800000000\tabularnewline
\hline 
6.6375 \%  & 900000000\tabularnewline
\hline 
7.0083 \%  & 1000000000\tabularnewline
\hline 
7.9108 \%  & 1100000000\tabularnewline
\hline 
7.9123 \%  & 1200000000\tabularnewline
\hline 
7.9166 \%  & 1300000000\tabularnewline
\hline 
7.9167 \%  & 1400000000\tabularnewline
\hline 
7.9178 \%  & 1500000000\tabularnewline
\hline 
7.9182 \%  & 1600000000\tabularnewline
\hline 
7.9182 \%  & 1700000000\tabularnewline
\hline 
7.9226 \%  & 1800000000\tabularnewline
\hline 
7.9285 \%  & 1900000000\tabularnewline
\hline 
7.93 \%  & 2000000000\tabularnewline
\hline 
7.93 \%  & 2100000000\tabularnewline
\hline 
7.9303 \%  & 2200000000\tabularnewline
\hline 
7.9596 \%  & 2300000000\tabularnewline
\hline 
7.9611 \%  & 2400000000\tabularnewline
\hline 
7.9612 \%  & 2500000000\tabularnewline
\hline 
7.9655 \%  & 2600000000\tabularnewline
\hline 
7.9655 \%  & 2700000000\tabularnewline
\hline 
7.9657 \%  & 2800000000\tabularnewline
\hline 
7.967 \%  & 2900000000\tabularnewline
\hline 
7.9671 \%  & 3000000000\tabularnewline
\hline 
7.9714 \%  & 3100000000\tabularnewline
\hline 
7.9773 \%  & 3200000000\tabularnewline
\hline 
7.9774 \%  & 3300000000\tabularnewline
\hline 
7.9788 \%  & 3400000000\tabularnewline
\hline 
7.9789 \%  & 3500000000\tabularnewline
\hline 
7.979 \%  & 3600000000\tabularnewline
\hline 
8.3014 \%  & 3700000000\tabularnewline
\hline 
8.3015 \%  & 3800000000\tabularnewline
\hline 
8.3044 \%  & 3900000000\tabularnewline
\hline 
8.3073 \%  & 4000000000\tabularnewline
\hline 
8.3088 \%  & 4100000000\tabularnewline
\hline 
8.3148 \%  & 4200000000\tabularnewline
\hline 
8.3191 \%  & 4300000000\tabularnewline
\hline 
8.3193 \%  & 4400000000\tabularnewline
\hline 
8.3207 \%  & 4500000000\tabularnewline
\hline 
8.3502 \%  & 4600000000\tabularnewline
\hline 
8.3578 \%  & 4700000000\tabularnewline
\hline 
8.3685 \%  & 4800000000\tabularnewline
\hline 
9.6965 \%  & 4900000000\tabularnewline
\hline 
97.2442 \%  & 5000000000\tabularnewline
\hline 
\end{tabular}
\par\end{centering}

\end{figure}

\begin{figure}
\begin{centering}
\begin{tabular}{|c|c|}
\hline 
Percentage of $r$ & $n$\tabularnewline
\hline 
\hline 
97.5846 \%  & 5100000000\tabularnewline
\hline
97.5905 \%  & 5200000000\tabularnewline
\hline 
97.5964 \%  & 5300000000\tabularnewline
\hline 
97.6038 \%  & 5400000000\tabularnewline
\hline 
97.6038 \%  & 5500000000\tabularnewline
\hline 
97.6334 \%  & 5600000000\tabularnewline
\hline 
97.6349 \%  & 5700000000\tabularnewline
\hline 
97.6393 \%  & 5800000000\tabularnewline
\hline 
97.6408 \%  & 5900000000\tabularnewline
\hline 
97.641 \%  & 6000000000\tabularnewline
\hline 
97.6511 \%  & 6100000000\tabularnewline
\hline 
97.6526 \%  & 6200000000\tabularnewline
\hline 
97.6528 \%  & 6300000000\tabularnewline
\hline 
97.9265 \%  & 6400000000\tabularnewline
\hline 
98.9342 \%  & 6500000000\tabularnewline
\hline 
98.9753 \%  & 6600000000\tabularnewline
\hline 
98.9769 \%  & 6700000000\tabularnewline
\hline 
98.9824 \%  & 6800000000\tabularnewline
\hline 
98.9828 \%  & 6900000000\tabularnewline
\hline 
98.9931 \%  & 7000000000\tabularnewline
\hline 
98.9946 \%  & 7100000000\tabularnewline
\hline  
\end{tabular}
\par\end{centering}

\end{figure}

\begin{figure}
\begin{centering}
\includegraphics[scale=0.5]{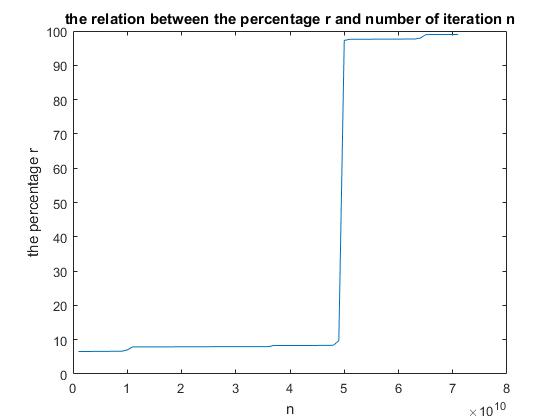}
\par\end{centering}

\caption{The graph of the percentage of $r$ and $n$}\label{fig:rngraph}

\end{figure}

When the program finished, it displayed the result:

r = 3.680103453346938e-04 (this is the area of the initial box, $100\%$ covered)

n = 7.180439126000000e+09 (total number of iterations needed)

min = 0.097626517574902 (the minimal area convex hull)

xx1 = 0.025097656250000

yy1 = 0.002578125000000

xx2 = 0.065327148437500

yy2 = 0.005422070312500

app = 0.079894832702377 (the 5 coordinates for the optimal configuration)

\subsection{Matlab program in Theorem \ref{thm2}}

Program cvhull2 below calculates the convex hull of the given configuration.

\subsubsection{the area of convex hull fucntion}

\begin{lstlisting}
function g = codecvhul2(xx)
%cvhul2 is the area of convex hull function
%Input xx is a vector of 5 parameters x1,y1,x2,y2, and alpha
x1=xx(1);
y1=xx(2);
x2=xx(3);
y2=xx(4);
alpha=xx(5);
global v  %set v is a width of rectangle
%create the  500-gon (F') 
t = linspace(0,999*pi/500,500); 
x=(sec(pi/500)*cos(t)/(2*pi));
y=(sec(pi/500)*sin(t)/(2*pi));
A=[x;y];
%create the vertex of a rectangle 
xR1=x1-(1/2-v)/2;
yR1=y1+v/2;
xR2=x1+(1/2-v)/2;
yR2=y1+v/2;
xR3=x1+(1/2-v)/2;
yR3=y1-v/2;
xR4=x1-(1/2-v)/2;
yR4=y1-v/2;
%create the vertex of an equilateral triangle
xQ1=x2+(sqrt(3)/9)*cos(alpha);
yQ1=y2+(sqrt(3)/9)*sin(alpha);
xQ2=x2+(sqrt(3)/9)*cos(alpha+(2*pi)/3);
yQ2=y2+(sqrt(3)/9)*sin(alpha+(2*pi)/3);
xQ3=x2+(sqrt(3)/9)*cos(alpha+(4*pi)/3);
yQ3=y2+(sqrt(3)/9)*sin(alpha+(4*pi)/3);
%B is the matrix of points of a rectangle
%C is the matrix of points of an equilateral triangle 
%D is the matrix of points of all objects
B = [xR1 xR2 xR3 xR4;yR1 yR2 yR3 yR4];
C = [xQ1 xQ2 xQ3;yQ1 yQ2 yQ3];
D = [A B C];
format long
%g is the area of convex hull of D 
[p,g]=convhull(D(1,:),D(2,:));
end
\end{lstlisting}

\subsubsection{the main program}
The function below finds set $W$ which satisfy the condition \ref{eq:cond2} in the proof of Theorem \ref{thm2}.\\

\begin{lstlisting}
(* ::Package:: *)

n=2000;% set the number of iteration
xx0=[0.013,0.01,0.051,0.011,0.08];% initial points for 5 parameters
global v
    x0=xx0;
    f=zeros(1,n); % matrix of function f
    fc=zeros(1,n); % matrx of function fc
    d=zeros(1,n); % matrix of grid length d
    w=zeros(1,n); % the width of rectangle
w(1)=0.0002; % w(1)
k=0;
for i=1:n
    v=w(i);
    if v>=0.25
       break;
    else
       x0=fminsearch(@codecvhul2,x0); % find new initial point
       f(i)=feval(@codecvhul2,x0); % minimum value of w(i)
       d(i)=(0.09763-f(i))/0.318; % maximum d(i) 
       fc(i)=f(i)+0.318*d(i); % f(w)+cd
       if fc(i)<=0.09763 % f(w)+cd<=B
          if i>=2 && w(i-1)+d(i-1)<w(i)-d(i) 
                   break;
          end
    fprintf('w(i) = % .7f, f(w(i)) = % .7f, d(i) = % .5e, w(i)-d(i) = % .8f,
    w(i)+d(i) = % .8f \n',v,f(i),d(i),v-d(i),v+d(i));
             k=k+1; % the number of real iteration
        end
            w(i+1)=w(i)+1.85*d(i); % to defind next point
    end
end
\end{lstlisting}

\subsubsection{the results}

Set $W$ satisfying the condition \ref{eq:cond2} in the proof of Theorem \ref{thm2}.

\includepdf[pages=-]{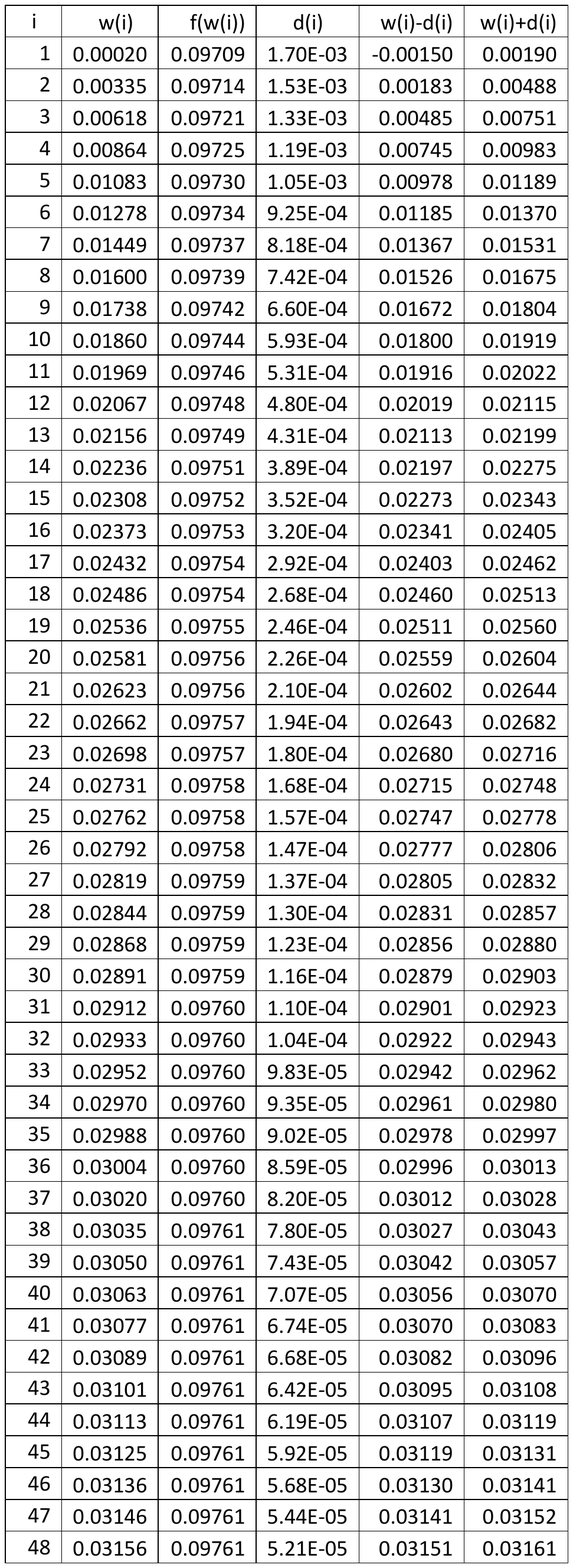}



\end{document}